\documentclass[11pt]{amsart}
\usepackage[margin=1in]{geometry}

\usepackage{graphicx}
\usepackage[active]{srcltx}
\usepackage{amsmath}
\usepackage{amscd}
\usepackage{tikz-cd}
\usepackage{amssymb}
\usepackage[mathscr]{euscript}
\usepackage{enumerate}
\usepackage{hyperref}
\usepackage{times}

\DeclareMathOperator{\End}{End}

\DeclareMathOperator{\Ann}{Ann}
\DeclareMathOperator{\Ext}{Ext}

\DeclareMathOperator{\Ass}{Ass}

\DeclareMathOperator{\Mod}{Mod}
\DeclareMathOperator{\cd}{cd}
\DeclareMathOperator{\ara}{ara}

\newcommand{\ZZ}{\mathbb{Z}}

\newcommand{\D}{\mathscr{D}}

\newcommand{\fm}{\mathfrak{m}}

\theoremstyle{plain}
\newtheorem{thm}{Theorem}[section]
\newtheorem*{thm*}{Theorem}
\newtheorem{bigthm}{Theorem}
\newtheorem{prop}[thm]{Proposition}
\newtheorem{lem}[thm]{Lemma}

\theoremstyle{definition}

\newtheorem{conjecture}[thm]{Conjecture}

\theoremstyle{remark}
\newtheorem{remark}[thm]{Remark}
\newtheorem{example}[thm]{Example}

\newtheorem{question}[thm]{Question}

\numberwithin{equation}{thm}

\begin{document}
\title{Annihilators of $\D$-modules in mixed characteristic}
\author{Rankeya Datta \and Nicholas Switala \and Wenliang Zhang}
\address{Department of Mathematics, Michigan State University, 619 Red Cedar Road, C212 Wells Hall, East Lansing, MI 48824}
\email{dattaran@msu.edu}
\address{Department of Mathematics, Statistics, and Computer Science \\ University of Illinois at Chicago \\ 322 SEO (M/C 249) \\ 851 S. Morgan Street \\ Chicago, IL 60607}
\email{nswitala@uic.edu, wlzhang@uic.edu}

\thanks{The second author gratefully acknowledges NSF support through grant DMS-1604503. The third author is partially supported by the NSF through grant DMS-1606414 and CAREER grant DMS-1752081.}

\begin{abstract}
Let $R$ be a polynomial or formal power series ring with coefficients in a DVR $V$ of mixed characteristic with a uniformizer $\pi$. We prove that the $R$-module annihilator of any nonzero $\D(R,V)$-module is either zero or is generated by a power of $\pi$. In contrast to the equicharacteristic case, nonzero annihilators can occur; we give an example of a top local cohomology module of the ring $\mathbb{Z}_2[[x_0, \ldots, x_5]]$ that is annihilated by $2$, thereby answering a question of Hochster in the negative. The same example also provides a counterexample to a conjecture of Lyubeznik and Yildirim.
\end{abstract}

\maketitle

\section{Introduction}\label{intro}
In \cite{HunekeProblemsLocalCohomology}, Huneke discussed 4 basic problems concerning local cohomology; these problems have guided the developments in the study of local cohomology modules for over two decades. As mentioned in the introduction of \cite{HunekeProblemsLocalCohomology}, ``We will find all of these problems are connected with another question: what annihilates the local cohomology?" More concretely, Hochster's \cite[Question 6]{HochsterSurveyLocalCohomology} asks the following:
\begin{question}
\label{Hochster question annihilator}
Is the top local cohomology module of a local Noetherian domain with support in a given ideal faithful? That is, if $R$ is a local Noetherian domain and $I \subseteq R$ is an ideal of cohomological dimension $\delta$, is $H^{\delta}_I(R)$ a faithful $R$-module?
\end{question}

Here $H^{\delta}_I(R)$ being faithful amounts to $\Ann_R(H^{\delta}_I(R))=(0)$. When $R$ is a regular local ring containing a field, Question \ref{Hochster question annihilator} has a positive answer; this was stated in \cite[Theorem 1.1]{LynchAnnLocalCohomology} and was attributed to Hochster and Huneke \cite[Lemma 2.2]{HunekeKohCofinitenessVanishing} in characteristic $p > 0$ and to Lyubeznik in characteristic zero. As stated in \cite[page 543]{LynchAnnLocalCohomology}, when $R$ is a regular local ring of mixed characteristic, Question \ref{Hochster question annihilator} remained open. 

Question \ref{Hochster question annihilator} also stems from a conjecture made by Lynch in \cite{LynchAnnLocalCohomology}; Lynch conjectured that Question \ref{Hochster question annihilator} has a positive answer for all Noetherian local rings even without assuming the ring is a domain. In \cite[Example 3.2]{BahmanpourLynchConjecture}, a counterexample to Lynch's conjecture was found; note that the local ring in \cite[Example 3.2]{BahmanpourLynchConjecture} is not equidimensional.

The main purpose of this paper is to investigate Question \ref{Hochster question annihilator} for regular local rings of mixed characteristic. One of our main results classifies annihilators of $\D$-modules as follows.

\begin{bigthm}[Theorem \ref{annihilators are powers of p}]
Let $R=V[[x_1,\dots,x_n]]$ or $R=V[x_1, \ldots, x_n]$ where $(V,\pi V)$ is a DVR of mixed characteristic $(0,p)$. Let $M$ be a nonzero $\D(R,V)$-module. Either $\Ann_R(M) = (0)$ or $\Ann_R(M) = (\pi^\ell)$ for some $\ell \geq 1$.
\end{bigthm}

Note that local cohomology modules of $R$ are primary examples of $\D(R,V)$-modules ({\it cf.} \S\ref{prelims}), and that investigating local cohomology modules from the $\D$-module viewpoint has proven fruitful over the years ({\it cf.} \cite{LyubeznikDMod}). 

We also answer Question \ref{Hochster question annihilator} in the negative in the case of regular local rings of mixed characteristic by considering a slight modification of Reisner's example.

\begin{bigthm}[Example \ref{main counterexample} and Remark \ref{arithmetic rank}]
Let $A=\ZZ[x_0,\dots,x_5]$ and $\fm=(2,x_0,\dots,x_5) \subseteq A$. Let $R$ be the $\fm$-adic completion of $A$, and let $I$ be the ideal of $R$ generated by the $10$ monomials
\[
\{x_0x_1x_2, x_0x_1x_3, x_0x_2x_4, x_0x_3x_5, x_0x_4x_5, x_1x_2x_5, x_1x_3x_4, x_1x_4x_5, x_2x_3x_4, x_2x_3x_5\}.
\]
Then $\cd(R,I)=\ara(I)=4$ and $\Ann_R(H^4_I(R))=(2)\neq (0)$, where $\ara(I)$ denotes the arithmetic rank of $I$.
\end{bigthm} 

This turns out to be a counterexample to a conjecture of Lyubeznik and Yildirim as well ({\it cf.} Remark \ref{counter-example to LY}).

The paper is organized as follows. In \S\ref{prelims}, we review some basics of the theory of $\D$-modules; in \S\ref{general annihilators}, we classify the annihilators of $\D(R,V)$-modules when $R$ is a ring of polynomials or formal power series over a DVR $(V,\pi V)$ of mixed characteristic $(0,p)$; in \S\ref{example}, we answer Question \ref{Hochster question annihilator} in the negative.

\subsection*{Acknowledgments} The authors thank Alberto Boix for comments on the paper. The second author thanks Tom Marley, who first interested him in these questions. The third author thanks Uli Walther for discussions. The authors also thank Benjamin Antieau and Kevin Tucker for helpful conversations. The authors are grateful to the referee for useful comments and suggestions.

\section{Preliminaries on $\D$-modules}\label{prelims}

We begin by fixing some conventions. All rings are assumed to have a unit element $1$. If $R$ is a commutative ring and $r \in R$ is an element, we denote by $(r)$ the principal ideal $rR \subseteq R$. All \emph{local} commutative rings are assumed to be Noetherian. When we say that $(V, \pi V, k)$ is a \emph{DVR of mixed characteristic $(0,p)$}, we mean that $V$ is a rank-one discrete valuation ring of characteristic zero whose maximal ideal is the principal ideal $(\pi) = \pi V$ generated by $\pi$, and whose residue field $k = V/\pi V$ has characteristic $p > 0$. If $\omega \in V$, we denote by $\nu_\pi(\omega)$ the $\pi$-adic valuation of $\omega$, that is, $\nu_\pi(\omega)$ is the exponent in the largest power of $\pi$ dividing $\omega$ (so $\nu_\pi(\omega) = 0$ if and only if $\omega$ is a unit in $V$). 

We now provide some necessary background material concerning $\D$-modules. If $S$ is any commutative ring and $A \subseteq S$ is a commutative subring, then the ring $\D(S,A)$ of $A$-linear differential operators on $S$, a subring of $\End_A(S)$, is defined recursively as follows \cite[\S 16]{EGAIV4}.  A differential operator $S \rightarrow S$ of order zero is multiplication by an element of $S$.  Supposing that differential operators of order $\leq j-1$ have been defined, $d \in \End_A(S)$ is said to be a differential operator of order $\leq j$ if, for all $s \in S$, the commutator $[d,s] \in \End_A(S)$ is a differential operator of order $\leq j-1$.  We write $\D^j(S)$ for the set of differential operators on $S$ of order $\leq j$ and set $\D(S,A) = \cup_j \D^j(S)$.  Every $\D^j(S)$ is naturally a left $S$-module.  If $d \in \D^j(S)$ and $d' \in \D^l(S)$, it is easy to prove by induction on $j + l$ that $d' \circ d \in \D^{j+l}(S)$, so $\D(S,A)$ is a ring. 

By a \emph{$\D(S,A)$-module}, we mean a \emph{left} module over the ring $\D(S,A)$. We denote by $\Mod_{\D(S,A)}$ the Abelian category of (left) $\D(S,A)$-modules. The ring $S$ itself has a $\D(S,A)$-module structure; using the quotient rule, we can give a $\D(S,A)$-module structure to the localization $S_f$ for every $f \in S$ in such a way that the natural localization map $S \rightarrow S_f$ is a map of $\D(S,A)$-modules. Using the \v{C}ech complex interpretation of local cohomology, it follows \cite[Example 2.1]{LyubeznikDMod} that the local cohomology modules $H^i_I(S)$ have $\D(S,A)$-module structures for all finitely generated ideals $I \subseteq S$ and all $i \geq 0$.

We will be concerned with the special case in which $S = A[[x_1, \ldots, x_n]]$ (resp. $S = A[x_1, \ldots, x_n]$) is a formal power series (resp. polynomial) ring with coefficients in $A$. In this case, we can describe explicitly the structure of the ring $\D(S,A)$: we have
\[
\D(S,A) = S \langle \partial_i^{[t]} \mid 1 \leq i \leq t, \, t \geq 1 \rangle
\]
by \cite[Thm. 16.11.2]{EGAIV4}, where $\partial_i^{[t]}$ denotes the differential operator $\frac{1}{t!} \frac{\partial^t}{\partial x_i^t}$ (which is well-defined even if the natural number $t$ is not invertible in $S$). If $a \in A$ is an element, then $a \D(S,A)$ is a two-sided ideal of $\D(S,A)$, and it follows from the displayed equality (with $A$ replaced by the quotient $A/(a)$) that
\[
\D(S,A)/a\D(S,A) \cong \D(S/(a), A/(a))
\]
as rings. In particular, a $\D(S/(a), A/(a))$-module is precisely a $\D(S,A)$-module annihilated by $a$.

We will need the following proposition in the sequel.

\begin{prop}\label{Boix zero annihilators}
Let $S = k[[x_1, \ldots, x_n]]$ or $k[x_1, \ldots, x_n]$ for some $n \geq 0$, where $k$ is a field. If $M$ is a nonzero $\D(S,k)$-module, then $\Ann_S(M) = (0)$.
\end{prop}

Proposition \ref{Boix zero annihilators} is essentially contained in the proof of \cite[Theorem 3.6]{BoixEghbali} which only treats local cohomology modules; it follows from \cite[Proposition 3.3]{BoixEghbaliCorrected}\footnote{The proof of Proposition 3.3 in \cite{BoixEghbali} is incomplete, and was corrected in \cite{BoixEghbaliCorrected}.} and \cite[Theorem 2.4]{BoixEghbali}. Also note that Proposition \ref{Boix zero annihilators} makes no assumption about the characteristic of $k$.


\section{The annihilator of a $\D(R,V)$-module}\label{general annihilators}

Throughout this section, $(V, \pi V, k)$ denotes a fixed DVR of mixed characteristic $(0,p)$, and $R$ denotes either the ring $V[[x_1, \ldots, x_n]]$ or $V[x_1, \ldots, x_n]$ for some $n \geq 0$. In either case, we denote by $\overline{R}$ the ring $R/(\pi)$, which is either a formal power series or polynomial ring over $k$. Our goal in this section is to classify the possible $R$-module annihilators of $\D(R,V)$-modules, and our main result is the following:

\begin{thm}\label{annihilators are powers of p}
Let $M$ be a nonzero $\D(R,V)$-module. Either $\Ann_R(M) = (0)$ or $\Ann_R(M) = \pi^\ell R$ for some $\ell \geq 1$.
\end{thm}

In order to prove Theorem \ref{annihilators are powers of p}, we begin with a classification of $\D(R,V)$-submodules of $R$.

\begin{thm}\label{D-submodules are powers of p}
Let $I \subseteq R$ be a nonzero $\D(R,V)$-submodule of $R$. There exists a natural number $\ell \geq 0$ such that $I = \pi^\ell R$.
\end{thm}

\begin{proof}
We write elements $f \in R$ in multi-index notation as follows: 
\[
f = \sum_{\mathbf{\beta} \in (\mathbb{Z}_{\geq 0})^n} \omega_{\mathbf{\beta}} \mathbf{x}^{\mathbf{\beta}},
\] 
where all $\omega_{\mathbf{\beta}} \in V$ and if $\beta = (\beta_1, \ldots, \beta_n) \in (\mathbb{Z}_{\geq 0})^n$, then $\mathbf{x}^{\mathbf{\beta}}$ denotes the monomial $x_1^{\beta_1} \cdots x_n^{\beta_n}$. When $R$ is a polynomial ring over $V$, we of course have $\omega_{\mathbf{\beta}} = 0$ for almost all $\beta \in (\mathbb{Z}_{\geq 0})^n$.

We prove the formal power series case first. Let $R = V[[x_1, \ldots, x_n]]$, let $I \subseteq R$ be a nonzero $\D(R,V)$-submodule, and let $f \in I$ be given. If there exists $\mathbf{\beta} \in (\mathbb{Z}_{\geq 0})^n$ such that $\omega_{\mathbf{\beta}}$ is a unit in $V$ (that is, $\nu_\pi(\omega_{\mathbf{\beta}}) = 0$), then $I = R = (\pi^0)$, since
\[
\partial_1^{[\beta_1]} \cdots \partial_n^{[\beta_n]}(\omega_{\mathbf{\beta}} x_1^{\beta_1} \cdots x_n^{\beta_n}) = \omega_{\mathbf{\beta}},
\]
and so, since $\partial_1^{[\beta_1]} \cdots \partial_n^{[\beta_n]}(f)$ (which belongs to $I$ by hypothesis) has a unit constant term, it is itself a unit in $R$. On the other hand, assume that for \emph{every} $f = \sum_{\mathbf{\beta} \in (\mathbb{Z}_{\geq 0})^n} \omega_{\mathbf{\beta}} \mathbf{x}^{\mathbf{\beta}} \in I$, we have $\nu_\pi(\omega_{\mathbf{\beta}}) > 0$ for all $\mathbf{\beta} \in (\mathbb{Z}_{\geq 0})^n$. Under this assumption, let $\ell$ be the minimal value of $\nu_\pi(\omega_{\mathbf{\beta}})$ among all $\omega_{\mathbf{\beta}}$ occurring as coefficients in any $f \in I$. This $\ell$ is a well-defined, nonzero natural number; we claim that $I = (\pi^\ell)$. It is clear that $I \subseteq (\pi^\ell)$. For the converse inclusion, choose $f = \sum_{\mathbf{\beta} \in (\mathbb{Z}_{\geq 0})^n} \omega_{\mathbf{\beta}} \mathbf{x}^{\mathbf{\beta}} \in I$ such that for some $\mathbf{\beta} \in (\mathbb{Z}_{\geq 0})^n$ we have $\nu_\pi(\omega_{\mathbf{\beta}}) = \ell$. Applying the differential operator $\partial_1^{[\beta_1]} \cdots \partial_n^{[\beta_n]}$, we obtain an element $g \in I$ whose \emph{constant} term is of the form $\pi^\ell$ times a unit in $V$. But by the minimality $\ell$, every other coefficient in $g$ is divisible by $\pi^\ell$; factoring out $\pi^\ell$, we can write $g$ as $\pi^\ell$ times a unit $h$ in $R$, from which it follows that $(\pi^\ell h)h^{-1} = \pi^\ell \in I$. Thus,  $I = (\pi^\ell)$ as claimed.

On the other hand, suppose that $R = V[x_1, \ldots, x_n]$. Again let $I \subseteq R$ be a nonzero $\D(R,V)$-submodule, and let $f \in I$ be given. Let $\gamma = (\gamma_1, \ldots, \gamma_n) \in (\mathbb{Z}_{\geq 0})^n$ be such that $\omega_{\mathbf{\gamma}} \mathbf{x}^{\mathbf{\gamma}}$ is the \emph{leading term} of $f$ with respect to the grlex term order. Then $\partial_1^{[\gamma_1]} \cdots \partial_n^{[\gamma_n]}(f) = \omega_{\mathbf{\gamma}} \in I$. Scaling by a unit if needed, we conclude that $\pi^{\nu_\pi(\omega_{\mathbf{\gamma}})} \in I$. Now let $\ell$ be the minimal $\pi$-adic valuation of any of the (grlex) leading coefficients of elements of $I$. Since $I \neq (0)$, $\ell$ is a natural number. By the preceding argument, $\pi^\ell \in I$. We claim that, conversely, $I \subseteq (\pi^\ell)$ and so the two are equal. Indeed, if $f \in I$ is not divisible by $\pi^\ell$, then some nonzero term $\omega_{\mathbf{\beta}} \mathbf{x}^{\mathbf{\beta}}$ has $\nu_\pi(\omega_{\mathbf{\beta}}) < \ell$. Let $\omega_{\mathbf{\beta}} \mathbf{x}^{\mathbf{\beta}}$ be the greatest term (under the grlex order) whose coefficient has $\pi$-adic valuation less than $\ell$. That is, $f$ can be written as $f=f_1+f_2$ where $\omega_{\mathbf{\beta}} \mathbf{x}^{\mathbf{\beta}}$ is the leading term of $f_1$ and $\pi^{\ell}$ divides $f_2$. Since $\pi^{\ell}
\in I$, this implies that $f_1\in I$ and its leading coefficient $\omega_{\mathbf{\beta}}$ is not divisible by $\pi^{\ell}$, a contradiction to the hypothesis on $\ell$. This completes the proof.
\end{proof}

Before we proceed to a proof of Theorem \ref{annihilators are powers of p}, we need the following lemma. In this lemma, and the rest of the results in this section, the proofs for the formal power series and polynomial cases are identical.

\begin{lem}\label{lem: saturation of ann is D-submodule}
Let $M$ be a $\D(R,V)$-module. Set $I=\Ann_R(M)$ and 
\[
J = (I:\pi^{\infty}) := \{a\in R\mid a\pi^m\in I\ {\rm for\ some\ integer\ }m \geq 0\}.
\] 
Then $J$ is a $\D(R,V)$-submodule of $R$.
\end{lem}

\begin{proof}
Let $p = \pi^\ell u$, for some $u \in V^{\times}$. Then it is easy to see that 
\[
J = (I:\pi^\infty) = (I:(\pi^\ell)^\infty) = (I:p^\infty),
\] 
and so, it suffices for us to show that $J = (I:p^\infty) = \{a\in R\mid ap^m\in I\ {\rm for\ some\ integer\ }m \geq 0\}$ is a $\D(R,V)$-submodule of $R$. It further suffices to show that $\partial_i^{[t]}(r)$ remains in $J$ for every $1\leq i\leq n$, $t\geq 1$, and $r\in J$.

Since $r\in J$, there is an integer $\ell$ such that $(p^\ell r)M=0$. First we consider the case when $t=1$. For each $m\in M$, we have
\[
0=\partial_i(p^\ell rm)=p^\ell \partial_i(rm)=p^\ell(\partial_i(r)m+r\partial_i(m))=p^\ell\partial_i(r)m,
\] 
which shows that $\partial_i(r)\in J$. Now an easy induction on $t$ shows that $\partial_i^t(r)\in J$ for all $t \geq 1$. Since $\partial_i^t=t!\partial_i^{[t]}$, it follows that $t!\partial_i^{[t]}(r)\in J$ for all $t \geq 1$. Since every integer coprime to $p$ is a unit in $R$, we have $p^{\nu_\pi(t!)}\partial_i^{[t]}(r)\in J$. By definition of $J = (I:p^\infty)$, there exists $m \geq 0$ such that $p^m(p^{\nu_\pi(t!)}\partial_i^{[t]}(r)) = p^{m+\nu_\pi(t!)}\partial_i^{[t]}(r) \in I$; again by definition of $J$, this means $\partial_i^{[t]}(r) \in J$.
\end{proof}

We will also need the following consequence of Proposition \ref{Boix zero annihilators}.

\begin{lem}\label{annihilators mod p}
Let $M$ be a nonzero $\D(R,V)$-module such that $\pi M = (0)$. Then $\Ann_R(M) = \pi R$.
\end{lem}

\begin{proof}
By hypothesis, $(\pi) \subseteq \Ann_R(M)$. For the converse inclusion, observe that since $M$ is annihilated by $\pi$, it has a natural structure of $\D(R,V)/\pi\D(R,V) = \D(R/(\pi), V/(\pi)) = \D(\overline{R}, k)$-module. Since $M \neq (0)$ and $k$ is a field, we have $\Ann_{\overline{R}}(M) = (0)$ by Proposition \ref{Boix zero annihilators}, so that $\Ann_R(M) \subseteq (\pi)$. 
\end{proof}

We are now ready to prove Theorem \ref{annihilators are powers of p}.

\begin{proof}[Proof of Theorem \ref{annihilators are powers of p}]
Let $I = \Ann_R(M)$ and assume that $I \neq (0)$. By Lemma \ref{lem: saturation of ann is D-submodule}, $J = (I:\pi^{\infty})$ is a nonzero $\D(R,V)$-submodule of $R$, so by Theorem \ref{D-submodules are powers of p}, we have $J = (\pi^e)$ for some natural number $e \geq 1$. In particular, $\pi^e \in J$, so for some $\ell \geq e$, $\pi^\ell \in I$ by definition of $J$. Assume that $\ell$ is the minimal integer such that $\pi^\ell \in I$. Since $M \neq (0)$, we must have $\ell \geq 1$. 

We will use induction on $\ell$ to show that $I = (\pi^\ell)$. The base case, $\ell = 1$, is precisely Lemma \ref{annihilators mod p}. If $\ell \geq 2$, we consider $\pi M$. The definition of $\ell$ implies that $\pi M\neq (0)$. Since $\pi \in V$, the module $\pi M$ is naturally a $\D(R,V)$-submodule of $M$. Applying the induction hypotheses to $\pi M$ (which is annihilated by $\pi^{\ell-1}$), by minimality of $\ell$ we have $\Ann_R(\pi M)=(\pi^{\ell-1})$. It follows immediately that $\Ann_R(M)=(\pi^{\ell})$.  
\end{proof}

Theorem \ref{annihilators are powers of p}, in conjunction with Theorem \ref{D-submodules are powers of p},  has the following interpretation: the $R$-module annihilators of $\D(R,V)$-modules are precisely the $\D(R,V)$-submodules of $R$.


\section{Reisner's example in mixed characteristic}\label{example}
The main purpose of this section is to produce an example of a top local cohomology module of a regular local ring of mixed characteristic that has nonzero annihilator. Such an example provides a negative answer to Question \ref{Hochster question annihilator} even in the special case of regular rings. We will begin with the following observation.

\begin{prop}
\label{integer torsion stays}
Let $A=\ZZ[x_1,\dots,x_n]$ for some $n \geq 0$ and let $I \subseteq A$ be a monomial ideal.
\begin{enumerate}[(a)]
\item For all $\ell \geq 1$, let $I_\ell$ be the ideal generated by the $\ell$-th powers of a given set of monomial generators for $I$. Suppose that $0 \neq \alpha \in \ZZ$ is such that $\alpha$ annihilates $\Ext^j_A(A/I,A)$ for some $j \geq 0$. Then $\alpha$ annihilates $\Ext^j_A(A/I_\ell,A)$ for all $\ell \geq 1$.
\item If $\alpha$ is as in part (a), then $\alpha$ annihilates $H^j_I(A)$; therefore, if $H^j_I(A) \neq (0)$ and such a nonzero $\alpha$ exists, then $\Ann_A(H^j_I(A))$ is a nonzero proper ideal of $A$.
\end{enumerate}
\end{prop}

\begin{proof}
Part (b) follows immediately from part (a) since $H^j_I(A)\cong \varinjlim_{\ell}\Ext^j_A(A/I_{\ell},A)$. To prove part (a), let $\ell \geq 1$ be given and consider the ring $A_{\ell}=\ZZ[x^{\ell}_1,\dots,x^{\ell}_n]$. Since $\alpha$ is an integer, $I$ (and therefore $I_{\ell}$) is generated by monomials, and $\alpha\Ext^j_A(A/I,A)=(0)$, it follows that $\alpha\Ext^j_{A_\ell}(A_{\ell}/I_{\ell},A_{\ell})=(0)$. It is clear that $A$ is a free $A_{\ell}$-module, and hence
\[
\Ext^j_A(A/I_{\ell},A)\cong \Ext^j_{A_\ell}(A_{\ell}/I_{\ell},A_{\ell})\otimes_{A_{\ell}}A
\]
as $A$-modules. It follows immediately that $\alpha\Ext^j_A(A/I_{\ell},A)=(0)$, completing the proof.
\end{proof}

For the rest of this section, we consider the following example, considered by Reisner in \cite[Remark 3]{Reisner} and associated with a minimal triangulation of the real projective plane. Let $A=\ZZ[x_0,\dots,x_5]$ and let $I$ be the ideal of $A$ generated by the $10$ monomials
\[
\{x_0x_1x_2, x_0x_1x_3, x_0x_2x_4, x_0x_3x_5, x_0x_4x_5, x_1x_2x_5, x_1x_3x_4, x_1x_4x_5, x_2x_3x_4, x_2x_3x_5\}.
\]
It is well-known that $\cd(I,A) \leq 4$ ({\it cf.} \cite{SchmittVogel}), {\it i.e.} $H^j_I(A)=(0)$ for $j\geq 5$. As we will see in the next example $H^4_I(A)\neq 4$ and hence $H^4_I(A)$ is a top local cohomology module. This allows us to give a negative answer to Question \ref{Hochster question annihilator} as follows.

\begin{example}\label{main counterexample}
Let $A$ and $I$ be as in the previous paragraph, and let $\fm=(2,x_0,\dots,x_5)$. It is straightforward to check\footnote{This can be seen directly from Hochster's formula; alternatively, as indicated by Lyubeznik in \cite{LyubeznikMonomialSupport}, it can be computed using the Taylor resolution of $A/I$.} that $\Ext^4_A(A/I,A)\cong A/\fm$, which is annihilated by $2$. As in Proposition \ref{integer torsion stays}, for all $\ell \geq 1$, let $I_{\ell}$ be the ideal generated by the $\ell$-th powers of the displayed monomial generators of $I$. Since $\Ext^4_A(A/I,A)\cong A/\fm$, it follows from the proof of Proposition \ref{integer torsion stays} that
\[
\Ext^4_A(A/I_{\ell},A) \cong A/(2,x^{\ell}_0,\dots,x^{\ell}_5)
\] 
for all $\ell \geq 1$. It follows from \cite[Theorem 1]{LyubeznikMonomialSupport}\footnote{In Lyubeznik's theorem, the ring was assumed to be a local ring containing a field. However, since each $\Ext^4_A(A/I_{\ell},A)$ is annihilated by 2 and supported only in the maximal ideal $\fm$, it is naturally a module over the local ring $\bar{R}$; hence Lyubeznik's theorem is applicable.} that the transition map $\Ext^4_A(A/I_{\ell},A)\to \Ext^4_A(A/I_{\ell+1},A)$ is injective for each $\ell\geq 1$. Therefore $H^4_I(A)\neq 0$ and is supported only in the maximal ideal $\fm$. It follows from Proposition \ref{integer torsion stays} that $2 \cdot \Ext^4_A(A/I_{\ell},A)=(0)$ for all $\ell \geq 1$; and hence $2
\cdot H^4_I(A)=(0)$. Now consider the $\fm$-adic completion $\widehat{A}^{\fm}$ of $A$, a complete unramified regular local ring. We again write $I$ for the ideal $I\widehat{A}^{\fm} \subseteq \widehat{A}^{\fm}$. (Concretely, the ring $\widehat{A}^{\fm}$ is isomorphic to $\mathbb{Z}_2[[x_0, \ldots, x_5]]$, where $\mathbb{Z}_2$ denotes the ring of $2$-adic integers, and $I\widehat{A}^{\fm}$ is generated by the same ten monomials.) Local cohomology commutes with the flat base change $A \rightarrow \widehat{A}^{\fm}$, so $H^4_I(\widehat{A}^{\fm}) \neq (0)$ and $H^j_I(\widehat{A}^{\fm}) = (0)$ for all $j > 4$. Consequently $H^4_I(\widehat{A}^{\fm})$ is a top local cohomology module that is annihilated by $2$. In particular, $\Ann_{\widehat{A}^\fm}(H^4_I(\widehat{A}^{\fm})) = (2)$ by Theorem \ref{annihilators are powers of p}.
\end{example}

\begin{remark}[Arithmetic rank of the ideal $I$]
\label{arithmetic rank}
It follows from \cite[Example 5)]{SchmittVogel} that the ideal $I$ can be defined up to radical by 4 elements
\[
\{x_0x_3x_5, x_0x_1x_3+x_0x_4x_5+x_2x_3x_5,x_0x_2x_4+x_1x_2x_5+x_1x_3x_4, x_0x_1x_2+x_1x_4x_5+x_2x_3x_4\}
\]
in both $A$ and $\widehat{A}^{\fm}$. Since it is shown in Example \ref{main counterexample} that $H^4_I(A)\neq 0$ and $H^4_I(\widehat{A}^{\fm})\neq 0$, the arithmetic rank of $I$ must be 4 in both $A$ and $\widehat{A}^{\fm}$ and $\cd(A,I) = \cd(\widehat{A}^{\fm},I)=4$.
\end{remark}

\begin{remark}
During the preparation of this paper, we learned that Hochster and Jeffries obtained the following result: {\it Let $(R,\fm)$ be a Noetherian local domain of characteristic $p$. Assume that the arithmetic rank of an ideal $I$ is the same as its cohomological dimension, which is denoted by $\delta$. Then $H^{\delta}_I(R)$ is faithful.}

The combination of Example \ref{main counterexample} and Remark \ref{arithmetic rank} shows that the mixed-characteristic analogue of the aforementioned Hochster-Jeffries result does not hold.  
\end{remark}


Hern\'{a}ndez, N\'{u}{\~n}ez-Betancourt, P\'{e}rez, and Witt also studied the module $H^4_I(\widehat{A}^{\fm})$, concluding \cite[Theorem 6.3]{HNPW1} that this module has zero-dimensional support while its injective dimension as an $\widehat{A}^{\fm}$-module is equal to $1$. We finish this section with a finer analysis of the structure of $H^4_I(\widehat{A}^{\fm})$, that in particular recovers this result of \cite{HNPW1}.

\begin{prop}
\label{Reisner example is injective hull mod 2}
Let $(R,\fm)=\widehat{A}^{\fm}$ denote the completion of $A$ at the maximal ideal $\fm=(2,x_0,\dots,x_5)$ and let $\bar{R}$ denote $R/(2)$. Then $H^4_I(R) \cong E_{\bar{R}}(\bar{R}/\fm)$ as $R$-modules.
\end{prop}

\begin{proof}
As we have seen in Example \ref{main counterexample} that the transition map $\Ext^4_A(A/I_{\ell},A)\to \Ext^4_A(A/I_{\ell+1},A)$ is injective for each $\ell\geq 1$. Since $(x^{\ell+1}_0,\dots,x^{\ell+1}_5):(x^{\ell}_0,\dots,x^{\ell}_5)=(x^{\ell+1}_0,\dots,x^{\ell+1}_5, x_0\cdots x_5)$ (which holds in both $A$ and $A/(2)$), it is straightforward to check that this transition map is given by 
\[\frac{A}{(2,x^{\ell}_0,\dots,x^{\ell}_5)}\xrightarrow{\cdot x_0\cdots x_5}\frac{A}{(2,x^{\ell+1}_0,\dots,x^{\ell+1}_5)}.\]
Therefore, 
\begin{align*}
H^4_I(A)&\cong \varinjlim_{\ell} \Ext^4_A(A/I_{\ell},A)\\
&\cong \varinjlim_{\ell} (\cdots\to  \frac{A}{(2,x^{\ell}_0,\dots,x^{\ell}_5)}\xrightarrow{\cdot x_0\cdots x_5}\frac{A}{(2,x^{\ell+1}_0,\dots,x^{\ell+1}_5)}\to \cdots)\\
&\cong H^6_{\fm}(A/(2))\\
&\cong H^6_{\fm}(\bar{R})
\end{align*} 
Therefore $H^4_I(R)\cong H^4_I(A)\otimes_AR\cong H^6_{\fm}(\bar{R})\cong E_{\bar{R}}(\bar{R}/\fm)$ as $R$-modules. This completes the proof.
\end{proof}

Inspired by a question of Hellus (\cite{HellusHabilitation}), Lyubeznik and Yildirim conjectured in \cite[Conjecture 1]{LyubeznikYildirimPAMS2018} that
\begin{conjecture}
\label{Lyubeznil-Yildirim}
Let $(R,\fm)$ be a regular local ring and $I$ be a non-zero ideal of $R$. If $H^j_I(R)\neq 0$, then $0\in \Ass_R(D(H^j_I(R)))$.
\end{conjecture}

\begin{remark}
\label{counter-example to LY}
It is observed in \cite{WaltherZhangSurvey} that our Proposition \ref{Reisner example is injective hull mod 2} provides an counterexample to Conjecture \ref{Lyubeznil-Yildirim}: let $R,I$ be the same as in Proposition \ref{Reisner example is injective hull mod 2}, then $D(H^4_I(R))\cong R/(2)$ and hence $0\notin \Ass_R(D(H^4_I(R)))$.
\end{remark}
\bibliographystyle{plain}

\bibliography{masterbib}

\begin{thebibliography}{10}

\bibitem{BahmanpourLynchConjecture}
Kamal Bahmanpour.
\newblock A note on {L}ynch's conjecture.
\newblock {\em Comm. Algebra}, 45(6):2738--2745, 2017.

\bibitem{BoixEghbali}
A.~Boix and M.~Eghbali.
\newblock Annihilators of local cohomology modules and simplicity of rings of
  differential operators.
\newblock \texttt{arXiv:1511.07780v3}, 2017.

\bibitem{BoixEghbaliCorrected}
A.~Boix and M.~Eghbali.
\newblock Correction to: {A}nnihilators of local cohomology modules and
  simplicity of rings of differential operators.
\newblock {\em Beitr. Alg. Geom.}, 59:685--688, 2018.

\bibitem{EGAIV4}
A.~Grothendieck and J.~Dieudonn\'{e}.
\newblock El\'{e}ments de g\'{e}om\'{e}trie alg\'{e}brique {IV}: \'{E}tude
  locale des sch\'{e}mas et des morphismes de sch\'{e}mas, quatri\`{e}me
  partie.
\newblock {\em Publ. Math. IH\'{E}S}, 32:5--361, 1967.

\bibitem{HellusHabilitation}
M.~Hellus.
\newblock Local cohomology and {M}atlis duality.
\newblock habilitation thesis, University of Leipzig, 2007.

\bibitem{HNPW1}
D.~Hern\'{a}ndez, L.~N\'{u}{\~n}ez-Betancourt, F.~P\'{e}rez, and E.~Witt.
\newblock Lyubeznik numbers and injective dimension in mixed characteristic.
\newblock {\em Trans. Amer. Math. Soc.}, 371(11):7533--7557, 2019.

\bibitem{HochsterSurveyLocalCohomology}
Melvin Hochster.
\newblock Finiteness properties and numerical behavior of local cohomology.
\newblock {\em Comm. Algebra}, 47(6):1--11, 2019.

\bibitem{HunekeProblemsLocalCohomology}
Craig Huneke.
\newblock Problems on local cohomology.
\newblock In {\em Free resolutions in commutative algebra and algebraic
  geometry ({S}undance, {UT}, 1990)}, volume~2 of {\em Res. Notes Math.}, pages
  93--108. Jones and Bartlett, Boston, MA, 1992.

\bibitem{HunekeKohCofinitenessVanishing}
Craig Huneke and Jee Koh.
\newblock Cofiniteness and vanishing of local cohomology modules.
\newblock {\em Math. Proc. Cambridge Philos. Soc.}, 110(3):421--429, 1991.

\bibitem{LynchAnnLocalCohomology}
Laura~R. Lynch.
\newblock Annihilators of top local cohomology.
\newblock {\em Comm. Algebra}, 40(2):542--551, 2012.

\bibitem{LyubeznikDMod}
G.~Lyubeznik.
\newblock Finiteness properties of local cohomology modules (an application of
  {$D$}-modules to commutative algebra).
\newblock {\em Invent. Math.}, 113:41--55, 1993.

\bibitem{LyubeznikMonomialSupport}
Gennady Lyubeznik.
\newblock On the local cohomology modules {$H^i_{{\mathfrak{a}}}(R)$} for
  ideals {${\mathfrak{a}}$} generated by monomials in an {$R$}-sequence.
\newblock In {\em Complete intersections ({A}cireale, 1983)}, volume 1092 of
  {\em Lecture Notes in Math.}, pages 214--220. Springer, Berlin, 1984.

\bibitem{LyubeznikYildirimPAMS2018}
Gennady Lyubeznik and Tu\u{g}ba Yildirim.
\newblock On the {M}atlis duals of local cohomology modules.
\newblock {\em Proc. Amer. Math. Soc.}, 146(9):3715--3720, 2018.

\bibitem{Reisner}
G.~Reisner.
\newblock Cohen-{M}acaulay quotients of polynomial rings.
\newblock {\em Adv. Math.}, 21(1):30--49, 1976.

\bibitem{SchmittVogel}
Thomas Schmitt and Wolfgang Vogel.
\newblock Note on set-theoretic intersections of subvarieties of projective
  space.
\newblock {\em Math. Ann.}, 245(3):247--253, 1979.

\bibitem{WaltherZhangSurvey}
Uli Walther and Wenliang Zhang.
\newblock Local cohomology -- an invitation.
\newblock arXiv:2106.09796.

\end{thebibliography}

\end{document}